\documentclass[a4paper,11pt,UKenglish]{article}
\usepackage[top=3.2cm, bottom=3.2cm, left=3.45cm, right=3.45cm]{geometry}
\usepackage[T1]{fontenc}
\usepackage[utf8]{inputenc}
\usepackage{lmodern}
\usepackage{amsmath,amssymb,amsthm,mathtools}
\usepackage{tikz}
\usepackage{microtype}
\usepackage[colorlinks=true,linkcolor=blue,citecolor=blue,urlcolor=blue]{hyperref}

\theoremstyle{plain}
\newtheorem{theorem}{Theorem}[section]
\newtheorem{proposition}[theorem]{Proposition}
\newtheorem{corollary}[theorem]{Corollary}
\newtheorem{lemma}[theorem]{Lemma}
\theoremstyle{remark}

\newtheorem{example}[theorem]{Example}
\newtheorem{definition}[theorem]{Definition}

\newcommand{\QQ}{\mathbb{Q}}
\newcommand{\Cay}{\mathrm{Cay}}
\newcommand{\Av}{\mathrm{Cay}}
\newcommand{\C}{\mathcal{C}}
\newcommand{\cl}[1]{\langle#1\rangle}
\newcommand{\clH}[1]{\cl{#1}_{\mathrm H}}
\DeclareMathOperator{\st}{\mathrm{st}}
\DeclareMathOperator{\img}{im}

\title{Finiteness theorems for Cayley permutation classes}
\author{Giulio Cerbai\\
  \normalsize Independent researcher\\
  \normalsize \texttt{giuliocerbai14@gmail.com}
  \and
  Anders Claesson\\
  \normalsize Department of Mathematics\\
  \normalsize University of Iceland\\
  \normalsize \texttt{akc@hi.is}}
\date{28 November 2026}

\begin{document}

\maketitle

\begin{abstract}
  We prove that every hereditary Cayley permutation class with finitely many
  simple members has a finite basis and an algebraic generating function.
  Given the simple members and basis of such a class, we construct a finite
  system of functional equations and use it to compute an annihilating
  polynomial for the generating function.
  \medskip

\end{abstract}

\thispagestyle{empty}
\section{Introduction}\label{intro}

A \emph{Cayley permutation} is either empty or a word of positive
integers that contains every number between one and its maximum
value. In other words, a Cayley permutation is an endofunction $\omega$ on
$[n]=\{1,2,\ldots,n\}$ such that $\img(\omega)=[m]$ for some $m\leq n$.
We write $\Cay_n$ for the set of Cayley permutations of length $n$ and
$\Cay=\bigcup_{n\geq 0}\Cay_n$. In addition, we write
$\Cay_+=\bigcup_{n\geq1}\Cay_n$ for the set of nonempty Cayley permutations.

Let $\omega$ be a word of positive integers. The \emph{standardization} of
$\omega$ is the word $\st(\omega)$ obtained by replacing the smallest letter
of $\omega$ by $1$, the next smallest by $2$, and so on, keeping repetitions
equal.
Note that $\st(\omega)$ is a Cayley permutation.

An occurrence of the classical pattern $\tau\in\Cay_k$ in $\omega\in \Cay_n$
is a subsequence $\alpha$ of $\omega$ such that $\st(\alpha)=\tau$. We write $\tau\leq \omega$
to indicate that $\tau$ occurs in $\omega$. This makes $\Cay$ into a poset. A
(hereditary)\kern0.2em\emph{Cayley permutation class} is an order
ideal $\C$ of $\Cay$; that is, $\tau\leq \omega\in\C \Rightarrow \tau\in\C$.
We say that $\omega$ \emph{avoids} $\tau$ if $\tau\nleq \omega$. Further, for a set
$X\subseteq\Cay$, we let $\Av(X)$ denote the class of Cayley
permutations avoiding each member of $X$.
The \emph{basis} of a class $\C$ is the set $\mathcal B$ of
$\leq$-minimal Cayley permutations not in $\C$; thus
$\C=\Av(\mathcal B)$.

In the companion paper~\cite{cerbai-claesson-simple}, we define simple
and semisimple Cayley permutations and prove a substitution decomposition
analogous to that of (ordinary) permutations. Here we use this decomposition to study
hereditary classes, recalling the necessary definitions and results
from that paper as they are needed. First, we highlight the relevant
results from the literature on permutations.

A permutation is \emph{simple}~\cite{AlbAtk05} if it has no proper
interval of length at least two, where an interval is a set of
consecutive positions whose corresponding values also form a
contiguous set of integers.
A special case of a general theorem by Schmerl and
Trotter~\cite{SchTro93} on indecomposable relational structures
says that every simple permutation of length $n\ge 2$ contains
a simple permutation of length $n-1$ or $n-2$. In other words, if
a permutation is simple then one may obtain a new simple permutation
by deleting at most two points. Permutations where a one-point
deletion is not sufficient are called \emph{exceptional}.
Schmerl and Trotter have classified the critically indecomposable
relational structures up to isomorphism of their skeletons.
And by reading the classification in terms
of permutations, Albert and Atkinson~\cite{AlbAtk05} showed that the
only exceptional permutations are the simple parallel alternations,
i.e.\ those of the form
\[
246\cdots (2m)135\cdots (2m-1),\quad m\ge 2,
\]
and their symmetries. They also showed that a (hereditary)
permutation class containing only finitely many simple permutations
is partially well ordered, finitely based, and has a readily
computable algebraic generating function.

The main goal of this paper is to establish analogous results for
Cayley permutations.
In Section~\ref{deletions}, we prove a deletion theorem and classify
the exceptional Cayley permutations.
Section~\ref{substitution-closed-classes} is devoted to substitution-closed
classes. We study closure under inflation and the resulting generating functions,
with and without the additional restriction that components at repeated quotient
values consist of one letter.
In Section~\ref{finite-simple-classes} we remove the hypothesis of substitution
closure and study general classes with finitely many simples.
We prove that every class with finitely many simple members is partially well ordered (Theorem~\ref{pwo}), has a finite basis (Theorem~\ref{finite-basis}), and has an
algebraic generating function (Theorem~\ref{algebraicity}). The corresponding
assertions for finitely many semisimple members follow as special cases.

\section{Exceptional Cayley permutations}\label{deletions}

Recall that every simple permutation of length $n\geq2$ contains
a simple permutation of length $n-1$ or $n-2$, by Schmerl and
Trotter~\cite{SchTro93}. Those for which no one-point deletion is
simple are called exceptional.
Albert and Atkinson~\cite{AlbAtk05} showed that they are precisely
the simple parallel alternations and their symmetries.

We wish to prove that both statements carry over to Cayley
permutations.
First, we recall some definitions from the companion
paper~\cite{cerbai-claesson-simple}.

Let $\omega$ be a word. We say that $\beta$ is a \emph{factor} of $\omega$ if there
are words $\alpha$ and $\gamma$ such that $\omega=\alpha\beta\gamma$.
Let $\omega=w_1\cdots w_n\in\Cay_n$. If $I=[a,b]\subseteq[n]$ is a
nonempty positional interval, write
\[
 \omega|_I=w_aw_{a+1}\cdots w_b
 \qquad\text{and}\qquad
 \omega(I)=\{w_a,w_{a+1},\ldots,w_b\}
\]
for the factor of $\omega$ at positions $I$ and for its set of values.
More generally, if $S=\{i_1<\cdots<i_r\}\subseteq[n]$, write
$\omega|_S=w_{i_1}\cdots w_{i_r}$ for the corresponding subword
and $\omega(S)$ for its values.
We say that $I$ is a \emph{Hertzsprung interval}, or \emph{H-interval},
of $\omega$ if there is a value interval $J=[c,d]$ such that
\begin{equation}\label{eq:h-interval}
 \omega(I)=J
 \quad\text{and}\quad
 \omega^{-1}(J)=I.
\end{equation}
A nonempty H-interval occupies consecutive positions and uses every value
in a consecutive range, which we call a \emph{value band}.
Additionally, it contains every copy in $\omega$ of every value in $J$ due
to the second equality, the \emph{saturation condition}.
Note that $[n]$ is always an H-interval, while a singleton
$[i,i]=\{i\}$ is an H-interval precisely when the value $\omega(i)$ occurs
nowhere else in $\omega$.
We also regard the empty set as an H-interval of every Cayley permutation,
with empty value band.

\begin{definition}
A Cayley permutation $\omega\in\Cay_n$ is \emph{semisimple} if every
H-interval of $\omega$ is empty, a singleton, or $[n]$;
these H-intervals are called \emph{trivial}. The word is \emph{primitive}~\cite{CCEGcayspec}
if adjacent letters are distinct. It is \emph{simple} if it is semisimple
and primitive, or if $\omega=11$.
\end{definition}

The empty word is semisimple and primitive, hence simple.
Both semisimplicity and simplicity restrict to the usual definition of
simplicity for permutations. Thus $1$, $11$, $12$, and $21$ are simple,
whereas $111$ is semisimple but not simple. A pair of equal adjacent
letters is called a \emph{flat step}.

To start, we take care of semisimple Cayley permutations that are
not primitive by showing that removing a flat step from a semisimple
Cayley permutation preserves semisimplicity.
Given $\omega=w_1\cdots w_n\in\Cay_n$ and a position $i\in [n]$, let
\[
  \omega\setminus i=\st(w_1\cdots w_{i-1}w_{i+1}\cdots w_n)
\]
denote the one-point deletion at $i$.

\begin{proposition}\label{flat-deletion}
  Let $\omega=w_1\cdots w_n\in\Cay_n$ be semisimple with $n\geq2$, and suppose $w_i=w_{i+1}$. Then
  $\omega\setminus(i+1)$ is a semisimple Cayley
  permutation of length $n-1$.
\end{proposition}

\begin{proof}
  Let $\omega'=w_1\cdots w_iw_{i+2}\cdots w_n$ and consider the
  order-preserving surjection
  \[
    g\colon[n]\to[n-1],
    \qquad
    g(j)=
    \begin{cases}
      j   & \text{if $j\leq i$},\\
      j-1 & \text{if $j>i$}.
    \end{cases}
  \]
  Since $w_i=w_{i+1}$, we have $\omega(j)=\omega'(g(j))$ for every $j$.
  Thus $\omega$ and $\omega'$ use the same values, so $\omega'$ is already a Cayley
  permutation and $\omega\setminus(i+1)=\omega'$.

  Let $I'$ be a nonempty H-interval of $\omega'$ with value band $J$.
  Its preimage $I=g^{-1}(I')$ is a nonempty positional interval, since
  $g$ is order-preserving and surjective. Moreover,
  \[
    \omega(I)=\omega'(I')=J,
    \qquad
    \omega^{-1}(J)=g^{-1}\bigl((\omega')^{-1}(J)\bigr)=I.
  \]
  Hence $I$ is an H-interval of $\omega$, so semisimplicity makes it a singleton
  or $[n]$. Taking its image under $g$ makes $I'=g(I)$ a singleton or
  $[n-1]$, proving that $\omega'$ is semisimple.
\end{proof}

A \emph{relational structure} is a set together with a collection
of binary relations on it.
Schmerl and Trotter~\cite{SchTro93} defined \emph{indecomposable}
relational structures and proved that a relational structure is
indecomposable if and only if it contains no nontrivial
\emph{module}.
Here, a module is a subset~$M$ of the structure
such that, for every relation $R$, every $i,j\in M$, and every $x\notin M$, we have
$iRx\Longleftrightarrow jRx$ and $xRi\Longleftrightarrow xRj$.
A module is nontrivial if it has at least two elements and is a proper subset
of the ground set.
The empty structure is allowed, and it is indecomposable since it has
no nontrivial module.

Given a Cayley permutation $\omega\in\Cay_n$, define the
relational structure
\[
  R(\omega)=\bigl([n],\,<_1,\,<_2\bigr)
\]
on the set of positions $[n]$, carrying two binary relations:
the position order $i<_1 j$ if $i<j$, and the value order
$i<_2 j$ if $\omega(i)<\omega(j)$.
The position order is a linear order, while the value order is
a strict weak order whose ties are the positions of equal value.
Both are irreflexive.

\begin{lemma}[\cite{cerbai-claesson-simple}]\label{modules}
  The modules of $R(\omega)$ are the H-intervals of $\omega$ together with its
  nonempty positional intervals with constant value.
\end{lemma}

\begin{theorem}[\cite{cerbai-claesson-simple}]\label{indecomposable}
  A Cayley permutation $\omega$ is simple if and only if the relational
  structure $R(\omega)$ is indecomposable.
\end{theorem}

Theorem~\ref{indecomposable} lets
us use the aforementioned result of Schmerl and Trotter.

\begin{theorem}\label{deletion}
  Every simple Cayley permutation of length $n\geq2$ contains a
  simple Cayley permutation of length $n-1$ or $n-2$, obtained
  by deleting one or two letters and standardizing.
\end{theorem}

\begin{proof}
  Let $\omega\in\Cay_n$ be simple.  For $n=2$, either deletion leaves $1$; for
  $n=3,4$, keeping any two positions and standardizing leaves a Cayley
  permutation of length two, which is simple.  We may therefore assume that
  $n\geq5$. By
  Theorem~\ref{indecomposable}, $R(\omega)$ is indecomposable.  Since its
  relations are binary and irreflexive, it follows by
  Schmerl and Trotter~\cite[Corollary~2.3]{SchTro93} that $R(\omega)$ has an
  indecomposable induced substructure on $n-1$ or $n-2$ points.
  Standardization preserves both the position order and the value order, so the
  substructure of $R(\omega)$ induced on a set $S$ of positions is isomorphic to
  $R(\st(\omega|_S))$; take $S$ with $R(\st(\omega|_S))$ indecomposable and
  $|S|\in\{n-1,n-2\}$.  Since $|S|\geq3$,
  Theorem~\ref{indecomposable} makes $\st(\omega|_S)$ simple.
\end{proof}

\begin{corollary}\label{semisimple-deletion}
  Every semisimple Cayley permutation of length $n\geq2$ has a semisimple
  pattern of length $n-1$ or $n-2$.
\end{corollary}

\begin{proof}
  If there is a flat step, apply Proposition~\ref{flat-deletion}.
  Otherwise the word is simple, and Theorem~\ref{deletion} applies.
\end{proof}

For $n\ge 2$, we call a simple Cayley permutation $\omega\in\Cay_n$
\emph{exceptional} if $\omega\setminus i$ is not simple for each
$i\in[n]$. An exceptional word has length
at least four, since every word of length at most two is simple.
In particular, it is primitive.

Proving that the parallel alternations are the only exceptional
Cayley permutations requires a bit more work.
Following Schmerl and Trotter~\cite{SchTro93}, we say that two
ordered pairs $(a,b)$ and $(c,d)$ of distinct elements of a
relational structure have the same \emph{type} if
$aR_ib\Leftrightarrow cR_id$ and $bR_ia\Leftrightarrow dR_ic$
for every relation $R_i$.
The equivalence classes of ordered pairs induced by the type
relation form the \emph{type skeleton}.

As an example, consider the Cayley permutation $\omega=231146561$.
In the type skeleton of $R(\omega)$, the type of the arc $i\to j$
records the comparison of the positions $i,j$ and the comparison
of the values $\omega(i),\omega(j)$. There are six possible types:
\begin{align*}
A=(<,<),\qquad B=(<,=),\qquad C=(<,>),\\[0.25em]
D=(>,<),\qquad E=(>,=),\qquad F=(>,>),
\end{align*}
where the first sign compares $i$ with $j$ and the second compares
$\omega(i)$ with $\omega(j)$. The following matrix displays the type skeleton
of $R(\omega)$. Its rows and columns are labelled $i_{\omega(i)}$, and the
entry in row $i_{\omega(i)}$ and column $j_{\omega(j)}$ is the type of the
ordered pair $(i,j)$:
\[
  \begin{array}{c|*{9}{c}}
     &1_2&2_3&3_1&4_1&5_4&6_6&7_5&8_6&9_1\\
    \hline
    1_2&-&A&C&C&A&A&A&A&C\\
    2_3&F&-&C&C&A&A&A&A&C\\
    3_1&D&D&-&B&A&A&A&A&B\\
    4_1&D&D&E&-&A&A&A&A&B\\
    5_4&F&F&F&F&-&A&A&A&C\\
    6_6&F&F&F&F&F&-&C&B&C\\
    7_5&F&F&F&F&F&D&-&A&C\\
    8_6&F&F&F&F&F&E&F&-&C\\
    9_1&D&D&E&E&D&D&D&D&-
  \end{array}
\]
For instance, the type of $(1,2)$ is $A=(<,<)$ since we have
positions $1<2$ and values $\omega(1)<\omega(2)$.
In the matrix, two ordered pairs have the same type exactly when
their cells carry the same letter, and the type skeleton is the
partition of the off-diagonal cells into the six letter classes.
Note also that the matrix makes the modules transparent: a set is
a module when its rows agree outside its own columns.
For instance, the set $\{1,2\}$ is a module since rows $1_2,2_3$
agree outside columns $1_2,2_3$; it is also an H-interval,
occupying consecutive positions and the saturated value band
$\{2,3\}$. Similarly, $[5,8]$ is a module as well as an H-interval
with saturated value band $\{4,5,6\}$. These are the maximal proper
H-intervals. The set $\{3,4\}$ is another module: it is a constant
positional interval, but not an H-interval since the value $1$
occurs again at position $9$ (hence saturation fails).

Now, in the language of Schmerl and Trotter, an indecomposable
structure is \emph{critically indecomposable} if no one-point
deletion of the structure is indecomposable. The two
authors~\cite[Theorem~5.1]{SchTro93} have fully classified the
type skeletons of critically indecomposable structures.
We will only need to count their types to show that no Cayley
permutations except for the parallel alternations are
exceptional.

\begin{lemma}\label{four-types}
  A critically indecomposable type skeleton has at most four types.
\end{lemma}

\begin{proof}
   The type skeleton of a critically indecomposable structure
   is, for some $r\geq2$, the type skeleton of one of the nine
   structures listed by Schmerl and Trotter~\cite[Section~4]{SchTro93}:
\[
\mathcal G_r, \mathcal P_r, \mathcal P'_r,
\mathcal B_r, \mathcal T^{(1)}_r, \mathcal T^{(2)}_r,
\mathcal T^{(3)}_r, \mathcal D_r, \mathcal D'_r.
\]
  The graph $\mathcal G_r$ carries one symmetric relation, so a pair is
  adjacent or not: two types.  The tournaments $\mathcal T^{(i)}_r$ carry one
  relation holding in exactly one direction on each pair: two types.  The
  posets $\mathcal P_r$ and $\mathcal P'_r$ and the oriented graph
  $\mathcal D_r$ carry one antisymmetric relation, so a pair is related one
  way, the other way, or not at all: three types.  In
  $\mathcal B_r=(V_r;P_r,P'_r)$ the two orders divide the pairs between them,
  $P_r$ taking the $(a_i,b_j)$ with $i\geq j$ and $P'_r$ the remaining
  $(a_i,b_j)$ together with the $(a_i,a_j)$ and the $(b_i,b_j)$; each pair is
  therefore related by exactly one of the two, in exactly one direction: four types.  The same holds in $\mathcal D'_r=(T^{(2)}_r;F_r,F'_r)$, where
  $F_r$ and $F'_r$ partition the arcs of the tournament $\mathcal T^{(2)}_r$.
  An isomorphism of type skeletons preserves the number of types, so the bound
  passes to every critically indecomposable type skeleton.
\end{proof}

\begin{lemma}\label{six-types}
  Let $\omega$ be a primitive Cayley permutation with a repeated value.  Then $R(\omega)$
  has six types.
\end{lemma}

\begin{proof}
  For $i<j$, the type of $(i,j)$ is determined by whether
  $\omega(i)<\omega(j)$, $\omega(i)=\omega(j)$, or $\omega(i)>\omega(j)$; reversing the ordered pair gives three
  corresponding types with the opposite position order. All three value
  comparisons occur. A repeated value supplies the equality case. If no pair
  satisfied $\omega(i)<\omega(j)$, then $\omega$ would be weakly decreasing and, by
  primitivity, strictly decreasing, contrary to the repeated value. The same
  argument with the order reversed supplies a pair with $\omega(i)>\omega(j)$. Thus all
  six types occur, and they are distinct.
\end{proof}

\begin{theorem}\label{exceptional}
  The exceptional Cayley permutations are exactly the simple parallel
  alternations and their symmetries.
\end{theorem}

\begin{proof}
  Let $\omega\in\Cay_n$ be exceptional.  Since every Cayley permutation of length at
  most two is simple, every one-point deletion of a word of length two or three
  is simple. Thus $n\geq4$, and $\omega$ is primitive by definition.  Suppose $\omega$ had a
  repeated value.  By
  Theorem~\ref{indecomposable}, $R(\omega)$ is indecomposable, and it is
  critically indecomposable.  Indeed, for each $j\in[n]$, its one-point
  deletion at $j$ is isomorphic to $R(\omega\setminus j)$, which has $n-1\geq3$
  points.  By Theorem~\ref{indecomposable}, this structure is
  indecomposable only if $\omega\setminus j$ is simple, and no such deletion is
  simple.  Lemma~\ref{four-types} then allows $R(\omega)$ at most four types, while
  Lemma~\ref{six-types} gives it six.
  So $\omega$ has no repeated value.  Being a permutation, $\omega$ is exceptional in the
  classical sense, hence a simple parallel alternation or one of its
  symmetries.  Conversely, the simple parallel alternations and their
  symmetries are permutations with no simple one-point deletion, and so are
  exceptional here.
\end{proof}

\begin{corollary}\label{simple-repeated-deletion}
  A simple Cayley permutation with a repeated value has a simple
  one-point deletion. The exceptional semisimple Cayley permutations
  of length at least two,
  defined by the absence of a semisimple one-point deletion, are also
  exactly the simple parallel alternations and their symmetries.
\end{corollary}

\begin{proof}
  The first assertion follows from Theorem~\ref{exceptional}.
  For the second, Proposition~\ref{flat-deletion} excludes every
  semisimple word with a flat step. Any remaining exceptional
  semisimple word is simple and has no simple one-point deletion, so
  Theorem~\ref{exceptional} applies. For permutations the two notions
  of simplicity agree.
\end{proof}

Two further consequences of Schmerl and Trotter's results~\cite[Theorems~2.2 and~5.9]{SchTro93} are useful
when retaining a specified pattern during deletion or extension.

\begin{corollary}
\label{simple-extension}
  Every simple Cayley permutation of length $n\geq6$ has a simple
  pattern of length $n-2$. Moreover, if a set of $k$ positions in a
  simple Cayley permutation of length $n$ induces a simple pattern,
  where $3\leq k\leq n-2$, then it can be extended by two positions to
  induce a simple pattern of length $k+2$.
\end{corollary}

\begin{proof}
  Let $\omega$ be a simple Cayley permutation. For the first
  assertion with $n\geq7$ and for the second assertion, apply the cited
  results to $R(\omega)$ and use
  Theorem~\ref{indecomposable}. Induced substructures correspond
  to standardized subwords, as in the proof of Theorem~\ref{deletion}.
  If $n=6$, then Theorem~\ref{deletion} gives a simple pattern of length
  four or five. In the latter case, the five-letter pattern is not
  exceptional by Theorem~\ref{exceptional}, since every parallel
  alternation has even length. It therefore has a simple one-point
  deletion of length four, which is also a pattern of $\omega$.
\end{proof}

Since the simple permutation $24153$ has no simple pattern of length
three the first assertion of the above corollary fails for $n=5$.

\section{Substitution-closed classes with finitely many simples}\label{substitution-closed-classes}

Theorem~\ref{deletion} and Corollary~\ref{semisimple-deletion} supply
finite certificates for the finiteness of the sets of simple and semisimple members.

\begin{corollary}\label{two-lengths}
  Let $\C$ be a class of Cayley permutations and $n\geq1$. If $\C$ has
  no simple member of length $n$ or $n+1$, then it has no simple member
  of length greater than $n$. The same assertion holds with semisimple
  in place of simple.
\end{corollary}

\begin{proof}
  Suppose, to the contrary, that $\C$ has a simple member of length
  greater than $n$, and choose one, $\omega$, of least length $\ell$.  By
  hypothesis, $\ell\geq n+2$.  By Theorem~\ref{deletion} there is a
  simple pattern $\nu\leq \omega$ of length $\ell-1$ or $\ell-2$. Note that
  $\nu\in\C$, because $\C$ is a class.  Thus $n\leq|\nu|<\ell$.  If $|\nu|$
  is $n$ or $n+1$, this contradicts the hypothesis.  If $|\nu|>n+1$, it
  contradicts the minimality of $\ell$. For semisimple members,
  the same argument uses Corollary~\ref{semisimple-deletion}.
\end{proof}

A class therefore has finitely many simple members if and only if two
consecutive lengths lack them. If membership in $\C$ is decidable, for
example from a finite basis, we can find all its simple members by searching
successive lengths and stopping when two consecutive lengths have none.
The search terminates whenever the set of simple members is finite; it
does not decide whether that set is finite. The same holds
for semisimple members.

We next recall inflation and the two substitution decompositions from
the companion paper~\cite{cerbai-claesson-simple}.
The inflation of Cayley permutations is defined analogously to how it is
defined for permutations,
but only letters appearing exactly once may be inflated by
nonconstant words. Otherwise,
positions carrying the same value could be expanded into
competing value bands, making the process ambiguous.
Formally, given $\sigma\in\Cay_k$
we associate a nonempty Cayley permutation
$\alpha_i$ with each position $i$, subject to the restriction
\begin{equation}\label{eq:constant-components}
  |\sigma^{-1}(\sigma(i))|>1
  \quad\Longrightarrow\quad \alpha_i=1^{r_i}\text{ for some }r_i\geq1.
\end{equation}
The \emph{inflation} $\sigma[\alpha_1,\ldots,\alpha_k]$ is the
Cayley permutation obtained by replacing each entry $\sigma(i)$ with
a factor order isomorphic to $\alpha_i$, while preserving the relative order of
the entries of $\sigma$. A singleton quotient fibre is replaced by a consecutive
value band; a repeated fibre is replaced by a single value.
The positional interval occupied by the entries
corresponding to $\alpha_i$ is called the $i$th \emph{part} of the inflation.

The decomposition into H-intervals must satisfy the stronger restriction
\begin{equation}\label{eq:repeat-restriction}
  |\sigma^{-1}(\sigma(i))|>1 \quad\Longrightarrow\quad \alpha_i=1.
\end{equation}
In this case, every nonsingleton part is an H-interval.
The inflations
\[
  1121=121[11,1,1],\qquad 111=11[1,11]
\]
satisfy~\eqref{eq:constant-components}, but not
\eqref{eq:repeat-restriction}.

An \emph{H-interval decomposition} of $\omega$ is a partition of its positions
into consecutive intervals, each of which is an H-interval or a single
position. Contracting its parts and standardizing gives a Cayley
permutation, called the \emph{quotient}.

Direct and skew sums are defined by
\[
  \alpha\oplus\beta=12[\alpha,\beta]
  \quad\text{and}\quad
  \alpha\ominus\beta=21[\alpha,\beta].
\]
A Cayley permutation that cannot be expressed as a direct sum of
shorter Cayley permutations is \emph{plus indecomposable}.
Analogously, it is \emph{minus indecomposable} if it cannot be
expressed as a skew sum of shorter Cayley permutations.

\begin{theorem}[\cite{cerbai-claesson-simple}]\label{substitution-thm}
  Let $\omega$ be a Cayley permutation of length at least two.
  Then there is a unique semisimple Cayley permutation $\sigma$ of length
  $k\ge 2$ such that
  \[
    \omega=\sigma[\alpha_1,\ldots,\alpha_k],
  \]
  for some nonempty Cayley permutations $\alpha_1,\ldots,\alpha_k$
  satisfying~\eqref{eq:repeat-restriction}.
  If the quotient $\sigma\notin\{12,21\}$, then the components
  $\alpha_i$ are uniquely determined.
  If $\sigma=12$, they are unique if we require $\alpha_1$ to be
  plus indecomposable. If $\sigma=21$, they are unique if we require
  $\alpha_1$ to be minus indecomposable.
\end{theorem}

\begin{theorem}[\cite{cerbai-claesson-simple}]\label{module-substitution-thm}
  Every Cayley permutation $\omega$ of length at least two has a unique
  simple quotient $\sigma$ of length $k\geq2$ such that
  \[
    \omega=\sigma[\alpha_1,\ldots,\alpha_k]
  \]
  for nonempty components satisfying~\eqref{eq:constant-components}.
  If $\sigma\notin\{11,12,21\}$, the components are unique. For $\sigma=11$
  they are unique on requiring $\alpha_1=1$; for $\sigma=12$ or $21$ they
  are unique on requiring $\alpha_1$ to be plus or minus indecomposable,
  respectively.
\end{theorem}

For $\omega=231146561$ from Section~\ref{deletions}, the
H-interval decomposition is
\[
  23\mid1\mid1\mid4656\mid1,
\]
which by contraction gives the semisimple quotient $21131$. Conversely,
$\omega$ is obtained as the inflation
\[
  231146561=21131[12,1,1,1323,1].
\]
Contracting the constant run in this quotient instead gives the simple
quotient $2131$ and the inflation
\[
  231146561=2131[12,11,1323,1].
\]

A class is \emph{substitution-closed} if
$\sigma[\alpha_1,\ldots,\alpha_k]$ lies in the class whenever $\sigma$ and
the $\alpha_i$ do. We also consider closure when only inflations
satisfying~\eqref{eq:repeat-restriction} are required to belong to the class.
Write $\cl{\C}$ for the substitution closure of $\C$, and
$\clH{\C}$ for its closure under this additional restriction.
Thus,
\[
  \C\subseteq\clH{\C}\subseteq\cl{\C}.
\]
These closures can differ. For instance, $\Av(111)$ is closed under the restricted
condition, whereas $11[11,1]=111$ prevents substitution closure.
That is,
\[
\Av(111)=\clH{\Av(111)}\subset\cl{\Av(111)}.
\]

Albert and Atkinson's~\cite{AlbAtk05} characterization of
substitution-closed permutation classes, also stated in
Brignall's~\cite{Bri10} survey, extends to both closure conditions.
We first describe how a classical pattern can occur in an inflation.

\begin{lemma}\label{patterns-through}
  Let $\omega=\sigma[\alpha_1,\ldots,\alpha_k]$ be an inflation and $\beta$ a
  nonempty Cayley permutation. Then $\beta\leq\omega$ if and only if
  \[
    \beta=\st(\sigma|_T)\bigl[\gamma_i:i\in T\bigr]
  \]
  for some $\emptyset\neq T\subseteq[k]$ and nonempty $\gamma_i\leq\alpha_i$.
  At repeated quotient letters the $\gamma_i$ are constant; if the original
  inflation satisfies~\eqref{eq:repeat-restriction}, they are $1$.
\end{lemma}

\begin{proof}
  Choose positions $S$ witnessing an occurrence of $\beta$ in $\omega$, and let $S_i$ be those
  in the $i$th part of the inflation. Put
  $T=\{i:S_i\neq\emptyset\}$ and $\gamma_i=\st(\omega|_{S_i})$ for $i\in T$.
  Then $\gamma_i\leq\alpha_i$, and $\gamma_i$ is constant whenever $\sigma(i)$ is
  repeated, since the corresponding part is constant. Under
  \eqref{eq:repeat-restriction}, that part is a single position and $\gamma_i=1$.
  Here $T$ is the set of indices whose corresponding part in the
  inflation contains at least one entry in the selected occurrence of
  $\beta$; and $\gamma_i$ is the standardization of all the entries in
  $S_i$ taking part in such occurrence.
  Now, restricting $\sigma$ to $T$ and inflating $\sigma|_T$
  by the $\gamma_i$ gives the same result as restricting $\omega$ to $S$.
  Indeed, within a part the selected letters have relative order $\gamma_i$.
  Between parts with distinct values of $\sigma$, they compare as those
  values do. Parts with equal values of $\sigma$ are constant factors
  carrying equal letters. Hence
  \[
    \beta=\st(\omega|_S)=\st(\sigma|_T)[\gamma_i:i\in T].
  \]
  This is a valid inflation. Indeed, every repeated letter of $\st(\sigma|_T)$
  comes from a repeated letter of $\sigma$ and is therefore inflated by a constant word.

  Conversely, choose an occurrence of each $\gamma_i$ in the $i$th part.
  Their union has the displayed standardization, by the same comparison
  of values, and thus witnesses an occurrence of $\beta$ in $\omega$.
\end{proof}

We shall compare a class with its substitution closure, but
first we need a lemma.

\begin{lemma}\label{localization}\label{module-localization}
  Let $\omega=\sigma[\alpha_1,\ldots,\alpha_k]$ be an inflation and $\beta$
  a simple Cayley permutation with $\beta\leq \omega$.  Then
  $\beta\leq\sigma$ or $\beta\leq\alpha_i$ for some $i$.
  If the inflation satisfies~\eqref{eq:repeat-restriction}, the conclusion
  holds for every semisimple $\beta$.
\end{lemma}

\begin{proof}
  The claim is immediate if $\beta$ is empty. Otherwise,
  write $\beta=\st(\sigma|_T)[\gamma_i:i\in T]$ as in
  Lemma~\ref{patterns-through}.  If $T=\{i\}$ then $\beta=\gamma_i\leq\alpha_i$ and we are done.
  Suppose instead that $|T|\ge 2$. A component $\gamma_i$ of length greater than one would occupy a
  nontrivial module of $R(\beta)$, which simplicity
  forbids; so $\gamma_i=1$ for each $i$ and $\beta=\st(\sigma|_T)\leq\sigma$.
  Under~\eqref{eq:repeat-restriction}, such a component occupies a proper
  H-interval of $\beta$ with more than one position, so semisimplicity
  suffices to exclude it.
\end{proof}

\begin{proposition}\label{substitution-closed}\label{module-closure}
  Let $\C$ be a class of Cayley permutations.
  \begin{enumerate}
    \item $\C=\cl{\C}$ if and only if each
      basis element of $\C$ is simple.
    \item $\C=\clH{\C}$ if and only if each basis
      element of $\C$ is semisimple.
  \end{enumerate}
\end{proposition}

\begin{proof}
  Let $\mathcal B$ be the basis of $\C$. Suppose first that $\C$ is
  substitution-closed and some $\beta\in\mathcal B$ is not simple.
  By Theorem~\ref{module-substitution-thm},
  $\beta=\sigma[\alpha_1,\ldots,\alpha_k]$ with $\sigma$ simple of length
  $2\leq k<|\beta|$ and the components obeying~\eqref{eq:constant-components}.
  The quotient and components are proper patterns of $\beta$, so they all
  lie in $\C$ by the minimality of a basis element. Substitution closure
  then gives $\beta\in\C$, a contradiction.

  Conversely, suppose every $\beta\in\mathcal{B}$ is simple, and let
  $\omega=\sigma[\alpha_1,\ldots,\alpha_k]$ with $\sigma,\alpha_i\in\C$.  If
  $\omega\notin\C$ then $\omega$ contains some $\beta\in\mathcal{B}$, necessarily simple; by
  Lemma~\ref{localization}, $\beta\leq\alpha_i$ for some $i$ or $\beta\leq\sigma$.
  Either possibility contradicts $\sigma,\alpha_i\in\C$. Hence $\omega\in\C$.

  For closure under~\eqref{eq:repeat-restriction}, the same argument uses
  semisimple quotients, Theorem~\ref{substitution-thm}, and the semisimple
  assertion of Lemma~\ref{localization}.
\end{proof}

\begin{lemma}\label{closure-simples}
  Let $\C$ be a class of Cayley permutations.
  \begin{enumerate}
    \item $\C$ and $\cl{\C}$ have the same simple members.
    \item $\C$ and $\clH{\C}$ have the same semisimple members.
  \end{enumerate}
\end{lemma}

\begin{proof}
  Let $\mathcal D$ be the set of all Cayley permutations $\omega$ such that
  every simple Cayley permutation $\sigma\leq \omega$ belongs to $\C$.
  It is a class containing $\C$, and
  Lemma~\ref{localization} shows that it is substitution-closed.
  Thus $\C\subseteq\cl{\C}\subseteq\mathcal D$.
  If $\sigma$ is a simple member of $\cl{\C}$, then
  $\sigma\in\mathcal D$ and $\sigma\leq\sigma$, so the defining condition
  of $\mathcal D$ gives $\sigma\in\C$. Conversely, every simple member
  of $\C$ belongs to $\cl{\C}$.
  Hence the two classes have the same simple members.
  The same argument applies to semisimple members and
  $\clH{\C}$, using the semisimple assertion of
  Lemma~\ref{localization}.
\end{proof}

\begin{corollary}\label{closure-basis}
  If the longest simple members of a class $\C$ have length $k$, then the basis
  elements of $\cl{\C}$ have length at most $k+2$.
  If the longest semisimple members of $\C$ have length $k$, the same
  bound holds for the basis elements of $\clH{\C}$.
\end{corollary}

\begin{proof}
  Let $\beta$ be a basis element of $\cl{\C}$, of length $\ell$.  As
  $\cl{\C}$ is substitution-closed, Proposition~\ref{substitution-closed}
  makes $\beta$ simple.  Suppose $\ell\geq k+3$.  By Theorem~\ref{deletion},
  $\beta$ properly contains a simple Cayley permutation $\sigma$ of length
  $\ell-1$ or $\ell-2$, hence of length greater than $k$; so $\sigma\notin\C$
  and, by Lemma~\ref{closure-simples}, $\sigma\notin\cl{\C}$.  But every
  proper pattern of a basis element lies in the class.  Hence $\ell\leq k+2$.
  For semisimple members and $\clH{\C}$, use the same
  argument with Corollary~\ref{semisimple-deletion} in place of
  Theorem~\ref{deletion}.
\end{proof}

Given $\omega\in\Cay$, let
\[
  \mu(\omega)=\#\{j:|\omega^{-1}(j)|=1\}
\]
denote the number of singleton fibres of $\omega$.

\begin{proposition}\label{closed-algebraic}
  Let $\C$ be a class closed under inflations satisfying
  \eqref{eq:repeat-restriction}, containing at least one nonempty word,
  with only finitely many semisimple members. Let $F$
  be the ordinary generating function for its nonempty members, and put
  \[
    \widetilde\Phi_{\C}(x,y)
    =\sum_{\substack{\sigma\in\C\text{ semisimple}\\|\sigma|\geq3}}
    x^{|\sigma|-\mu(\sigma)}y^{\mu(\sigma)}.
  \]
  For a word $\omega$, let $\delta_{\omega}$ be $1$ if $\omega\in\C$ and $0$ otherwise. Then
  \begin{equation}\label{eq:closed-algebraic}
    \widetilde\Phi_{\C}(x,F)=F-x-\delta_{11}x^2
      -\frac{(\delta_{12}+\delta_{21})F^2}{1+F}.
  \end{equation}
  In particular, the ordinary generating function of $\C$ is algebraic
  over $\QQ(x)$.
\end{proposition}

\begin{proof}
  Since $\C$ contains a nonempty word and is closed under classical pattern
  containment, it contains $1$. For an inflation satisfying~\eqref{eq:repeat-restriction},
  say $\omega=\sigma[\alpha_1,\ldots,\alpha_k]$, the quotient $\sigma$ and each
  component $\alpha_i$ are classical patterns of $\omega$. Thus, if $\omega\in\C$,
  then $\sigma,\alpha_1,\ldots,\alpha_k\in\C$. Conversely, if the quotient
  and all components belong to $\C$, the closure condition gives $\omega\in\C$.
  Lemma~4.3 of the companion
  paper~\cite{cerbai-claesson-simple},
  applied to membership in $\C$ with the singleton-fibre weight set to $1$,
  now gives~\eqref{eq:closed-algebraic}.

  To prove algebraicity, note that $\widetilde\Phi_{\C}$ is a polynomial because
  $\C$ has only finitely many semisimple members. Multiplying
  \eqref{eq:closed-algebraic} by $1+F$ and rearranging gives $P(x,F)=0$, where
  \[
    P(x,y)=(1+y)\bigl(y-x-\delta_{11}x^2-\widetilde\Phi_{\C}(x,y)\bigr)
      -(\delta_{12}+\delta_{21})y^2
    \in\QQ[x,y].
  \]
  Every monomial of $\widetilde\Phi_{\C}$ has total degree at least three, so the
  coefficient of $x^0y$ in $P$ is $1$. In particular, $P$ is a nonzero
  polynomial, and $P(x,F)=0$ proves that $F$ is algebraic over $\QQ(x)$.
  Including the empty word gives the class generating function $1+F$,
  which is therefore algebraic as well.
\end{proof}

\begin{proposition}\label{module-closed-algebraic}
  Let $\C$ be substitution-closed, contain a nonempty word,
  and have finitely many simple members. Let $F$ count its nonempty
  members and put
  \[
    \Phi_{\C}(x,y)=
      \sum_{\substack{\sigma\in\C\text{ simple}\\|\sigma|\geq3}}
        x^{|\sigma|-\mu(\sigma)}y^{\mu(\sigma)}.
  \]
  With the indicators $\delta_\omega$ as before,
  \[
    F=x+\delta_{11}\frac{x^2}{1-x}
       +\frac{(\delta_{12}+\delta_{21})F^2}{1+F}
       +\Phi_{\C}\!\left(\frac{x}{1-x},F\right).
  \]
  In particular, the generating function of $\C$ is algebraic.
\end{proposition}

\begin{proof}
  Use the unique decomposition in Theorem~\ref{module-substitution-thm}.
  If $11\in\C$, closure gives all nonempty constant words; otherwise
  only the constant word $1$ belongs to $\C$. This yields the first
  two terms. The sum contributions are as in
  Proposition~\ref{closed-algebraic}. A repeated position of a longer
  simple quotient may carry any nonempty constant run. Such a quotient
  contains $11$, so all these runs belong to $\C$, and its contribution
  is obtained by substituting $x/(1-x)$ at repeated positions and $F$
  at singleton positions. Finally, $\Phi_{\C}$ is a polynomial.
  Clearing denominators gives a nonzero polynomial equation for $F$;
  the corresponding equation has coefficient $1$ on its linear term
  in $F$ at the origin.
\end{proof}

For $\C=\Av(21)$, i.e.\ the set of weakly increasing Cayley permutations, the polynomial $\Phi_{\C}$ is zero and the equation
in Proposition~\ref{module-closed-algebraic}
reduces to $F=x/(1-x)+F^2/(1+F)$, giving $F=x/(1-2x)$.
The same result is obtained by Proposition~\ref{closed-algebraic},
where $\widetilde\Phi_{\C}(x,y)=x^3/(1-x)$ since the only semisimple
members of $\C$ are the constant words.

\begin{example}\label{rational-class}
  For $k\geq3$, let $\C_k=\Av(1^k,21)$. Its members are precisely
  \[
    1^{a_1}2^{a_2}\cdots m^{a_m},
    \qquad 1\leq a_i<k.
  \]
  Every constant word $1^{a_1}$ on the displayed list is semisimple, as is $12$.  Any other
  member with a run of length at least two has that run as a proper H-interval,
  while one with only singleton runs and length at least three begins with the
  proper H-interval $12$.
  Therefore, the nonempty semisimple members of $\C_k$ are
  \[
    1,11,\ldots,1^{k-1},\quad\text{and}\quad 12.
  \]
  Now, in a nonempty Cayley permutation of $\C_k$ the first value occurs
  between one and $k-1$ times, after which the
  remaining suffix, standardized, is again in $\C_k$. Thus, writing $G$
  for the ordinary generating function of $\C_k$, we have
  \[
    G(x)=1+(x+x^2+\cdots+x^{k-1})G(x),
  \]
  and hence
  \[
    G(x)=\frac{1}{1-x-x^2-\cdots-x^{k-1}}.
  \]
  In particular, $\Av(111,21)$ has generating function $1/(1-x-x^2)$.

  We obtain the same answer from Proposition~\ref{closed-algebraic}.
  The basis elements $1^k$ and $21$ are semisimple, so $\C_k$ is
  closed under inflations satisfying~\eqref{eq:repeat-restriction}. Here
  $\delta_{11}=\delta_{12}=1$, $\delta_{21}=0$, and
  $\widetilde\Phi_{\C_k}=\sum_{j=3}^{k-1}x^j$, since the constant quotients have no
  singleton fibres. Thus~\eqref{eq:closed-algebraic} simplifies to
  $F=(x+x^2+\cdots+x^{k-1})(1+F)$, as required.

  As this example illustrates, semisimple quotients are particularly
  convenient for multiplicity bounds.  Each value occurs at most $r$
  times in an inflation exactly when this holds for the quotient and
  every component. Note that the class $\C_k$ is not
  substitution-closed, since $11[1,1^{k-1}]=1^k$, although the
  quotient and both components belong to $\C_k$. Thus,
  Proposition~\ref{closed-algebraic} applies, whereas
  Proposition~\ref{module-closed-algebraic} does not.
\end{example}

\begin{example}\label{algebraic-class}
  Let $\C$ be the class avoiding the classical patterns
  \[
    111,\quad 121,\quad 212,\quad 231.
  \]
  The semisimple members of lengths one and two are $1,11,12,21$.
  We see directly that there are no semisimple members of length three or four, so
  Corollary~\ref{two-lengths} shows that these are all the nonempty semisimple
  members of $\C$.

  This class is not closed even under inflations satisfying
  \eqref{eq:repeat-restriction}, since $21,12,1\in\C$ but
  $21[12,1]=231\notin\C$. Thus, Proposition~\ref{closed-algebraic} does not apply.
  Instead, we use Theorem~\ref{substitution-thm} to partition the nonempty
  members by their unique semisimple quotient. Each quotient is a classical
  pattern of the member being decomposed and hence belongs to $\C$, so the
  only possibilities are $1,11,12,21$. Let $F$ be the ordinary generating
  function for the
  nonempty members of $\C$. The quotients $1$ and $11$ contribute $x$ and
  $x^2$, respectively.

  For the quotient $12$, the canonical decomposition is
  $\alpha\oplus\beta$, with $\alpha$ plus indecomposable.
  The class $\C$ is closed under direct sums because each forbidden pattern
  is plus indecomposable. If $I$ counts the nonempty plus indecomposable
  members, then $F=I+IF$, so $I=F/(1+F)$.
  Hence this quotient contributes $IF=F^2/(1+F)$.

  For the quotient $21$, the canonical decomposition is
  $\alpha\ominus\beta$, with $\alpha$ minus indecomposable.
  Since $\beta$ is nonempty, an increasing pair in $\alpha$, together with
  any letter of $\beta$, would form $231$. Thus $\alpha$ must avoid $12$. It is
  therefore weakly decreasing, and if it used more than one value it would be
  minus decomposable; hence it is constant.  Avoidance of $111$ leaves only
  $\alpha=1,11$.
  Conversely, with $\alpha=1$ or $11$ and $\beta\in\C$ nonempty,
  none of the forbidden patterns can occur across the two components.
  Thus the quotient $21$ contributes $(x+x^2)F$. Adding the four
  contributions gives
  \[
    F=x+x^2+\frac{F^2}{1+F}+(x+x^2)F.
  \]
  On writing $G=1+F$ for the ordinary generating function of $\C$, this
  simplifies to
  \[
    G=1+(x+x^2)G^2.
  \]
  A direct bijection also explains this equation.
  Avoidance of $111$, $121$, and $212$ says that every value occurs
  in one contiguous run of length one or
  two.  Contracting the runs gives a $231$-avoiding permutation, and conversely
  each entry of any such permutation may be replaced independently by a run of
  length one or two. Permutations avoiding $231$ are counted by
  the Catalan numbers; letting $\mathrm{Cat}$ denote their generating
  function, we obtain
  \[
    G(x)=\mathrm{Cat}(x+x^2)
    =\frac{1-\sqrt{1-4x-4x^2}}{2x(1+x)}.
  \]
  The second part of Example~\ref{algebraic-system-class} illustrates the profile construction
  of Theorem~\ref{algebraicity}
  for a class with one further simple member.
\end{example}

\section{General classes with finitely many simples}\label{finite-simple-classes}

We now remove the hypothesis of substitution closure and extend
Proposition~\ref{module-closed-algebraic} to any class. Namely, we show
in Theorem~\ref{finite-basis} and Theorem~\ref{algebraicity}
that every class with finitely many simple members has a finite
basis and an algebraic generating function. This finiteness hypothesis allows
infinitely many semisimple members; for example, $\Av(21)$ has simple
members $1$, $11$, and $12$, but contains every constant word.
The first step is to rule out
infinite antichains. A partially ordered set is \emph{partially well ordered}
if it contains neither an infinite strictly descending chain nor an infinite
antichain. For Cayley permutations, a strictly descending chain is necessarily
finite, since length decreases at each step.

The \emph{closure} of a set under a family of operations consists of all
elements obtained by applying those operations finitely many times, starting
with elements of the set. We use the following form of Higman's theorem,
stated by Albert and Atkinson~\cite{AlbAtk05}.

\begin{theorem}[Higman~{\cite{Hig52}}]\label{higman}
  Let $(P,\leq)$ be a partially ordered set and let
  $f_i:P^{n_i}\to P$, $1\leq i\leq r$, be finitely many operations of finite
  arity. Suppose each $f_i$ is order preserving in every argument; that is,
  for all $(a_1,\ldots,a_{n_i}),(b_1,\ldots,b_{n_i})\in P^{n_i}$,
  \[
    a_j\leq b_j\text{ for every }1\leq j\leq n_i
    \quad\Longrightarrow\quad
    f_i(a_1,\ldots,a_{n_i})\leq f_i(b_1,\ldots,b_{n_i}).
  \]
  Suppose also that
  \[
    a_j\leq f_i(a_1,\ldots,a_{n_i})
    \quad\text{for every }(a_1,\ldots,a_{n_i})\in P^{n_i}
    \text{ and }1\leq j\leq n_i.
  \]
  Then the closure of any finite subset of $P$ under these operations is
  partially well ordered.
\end{theorem}

The proof of the next theorem adapts Albert and Atkinson's~\cite{AlbAtk05}
approach to our settings. Gustedt~\cite{Gus99} proves a related well-quasi-ordering
result for substitution compositions of posets.

\begin{theorem}\label{pwo}
  A class $\C$ of Cayley permutations with only finitely many simple members is
  partially well ordered.
\end{theorem}

\begin{proof}
  Let $F$ be the finite set of simple members of $\C$ of length at
  least two other than $11$. Consider finite rooted ordered trees
  whose nodes have labels in $\{1,\ast\}\cup F$. The label determines
  the number of children: a node labelled $1$ has no children, a node
  labelled $\ast$ has one child, and a node labelled $\sigma\in F$
  has $|\sigma|$ children.

  We interpret these trees as words recursively. A single leaf
  represents the word $1$. Suppose the root has label $\sigma\in F$,
  with $k=|\sigma|$, and its child subtrees represent
  $\alpha_1,\ldots,\alpha_k$ in that order. Then the tree represents
  $\sigma[\alpha_1,\ldots,\alpha_k]$, provided the components at
  repeated quotient letters are constant, as in
  \eqref{eq:constant-components}. A tree whose root has label $\ast$
  represents $1^{r+1}$ if its child subtree represents $1^r$.
  Call a tree \emph{valid} if this interpretation is defined at every
  node. In particular, $1^r$ is represented by a chain of $r-1$ nodes
  labelled $\ast$ above a leaf.

  Every nonempty member of $\C$ has a valid representation, by
  induction on length. For length one, use a leaf. For a longer
  word, Theorem~\ref{module-substitution-thm} gives a simple quotient
  and shorter components. All are patterns of the original word
  and hence belong to $\C$. If the quotient is $11$, the word is
  constant and we use the chain above. Otherwise, the quotient
  belongs to $F$; attach a root with this label above the component
  trees supplied by the induction hypothesis.
  Figure~\ref{fig:pwo-tree} illustrates the representation for the
  word $231146561$ considered in Section~\ref{deletions}.

  Now consider all finite ordered trees with these labels and the
  prescribed numbers of children.
  We define an embedding order on trees that implies pattern containment of the corresponding words.
  Recursively, let $S\preccurlyeq T$
  if and only if at least one of the following conditions holds:
  \begin{enumerate}
    \item[(i)] The roots of $S$ and $T$ have the same label, and
      their ordered child subtrees $S_1,\ldots,S_k$ and
      $T_1,\ldots,T_k$ satisfy $S_i\preccurlyeq T_i$ for every $i$.
    \item[(ii)] There is a child of the root of $T$ whose rooted
      subtree $U$ satisfies $S\preccurlyeq U$.
  \end{enumerate}
  In particular, a
  leaf embeds in a leaf by (i), since there are no children to compare.
  Furthermore, repeated use of (ii) allows the root of $S$ to be matched with a
  node below the root of $T$.

  For each internal label, consider the operation of attaching a
  root with that label above the prescribed number of trees, in
  order. Condition (i) makes this operation order preserving, and
  (ii) ensures that every argument embeds in the resulting tree.
  These finitely many
  operations generate all the trees from the single leaf.
  Theorem~\ref{higman} therefore makes this set of trees partially
  well ordered. Its subset of valid trees is also partially well
  ordered.

  On valid trees, embedding implies containment of the represented
  words. We prove this by induction on the number of nodes of $T$.
  In case (ii), the word represented by $S$ is a pattern of the word
  represented by $U$, by induction. The word represented by $U$ is
  in turn a pattern of the word represented by $T$: at an inflation
  root it is one of the components, and at a root labelled $\ast$
  this follows from $1^r\leq1^{r+1}$. Transitivity gives the result.
  In case (i), two leaves both represent $1$. If the common root
  label is $\sigma\in F$, the induction hypothesis gives an occurrence
  of each component represented at a child of $S$ in the component
  represented at the corresponding child of $T$. Combining these
  occurrences gives an occurrence of the smaller inflation in the
  larger one.
  If the common label is $\ast$, let the child subtrees represent
  $1^r$ and $1^s$. The induction hypothesis gives $r\leq s$, and
  hence $1^{r+1}\leq1^{s+1}$ for the words represented by $S$ and $T$.
  Thus any infinite sequence of nonempty members of $\C$ has an
  earlier member contained in a later one, by choosing their
  representing trees and applying partial well order there.
  Adjoining the empty word preserves this property.
\end{proof}

\begin{figure}[htbp]
  \centering
  \begin{tikzpicture}[
      quotient/.style={draw,rounded corners=2pt,inner sep=4pt},
      leaf/.style={circle,draw,inner sep=2pt}]
    \node[quotient] (root) at (0,0) {$2131$};
    \node[quotient] (left) at (-3.8,-1.2) {$12$};
    \node[quotient] (constant) at (-2.1,-1.2) {$\ast$};
    \node[quotient] (right) at (1,-1.2) {$12$};
    \node[leaf] (last) at (3.8,-1.2) {$1$};
    \node[leaf] (left1) at (-4.3,-2.4) {$1$};
    \node[leaf] (left2) at (-3.3,-2.4) {$1$};
    \node[leaf] (constant1) at (-2.1,-2.4) {$1$};
    \node[leaf] (right1) at (-.6,-2.4) {$1$};
    \node[quotient] (middle) at (1.7,-2.4) {$212$};
    \node[leaf] (middle1) at (.6,-3.6) {$1$};
    \node[leaf] (middle2) at (1.7,-3.6) {$1$};
    \node[leaf] (middle3) at (2.8,-3.6) {$1$};
    \draw (root) -- (left) (root) -- (constant)
      (root) -- (right) (root) -- (last);
    \draw (left) -- (left1) (left) -- (left2)
      (constant) -- (constant1);
    \draw (right) -- (right1) (right) -- (middle);
    \draw (middle) -- (middle1) (middle) -- (middle2)
      (middle) -- (middle3);
  \end{tikzpicture}
  \caption{A valid tree representing
    $231146561=2131[12,11,1323,1]$.
    The $\ast$ node represents $11$, and the right-hand $12$ node
    represents $12[1,212]=1323$.}
  \label{fig:pwo-tree}
\end{figure}
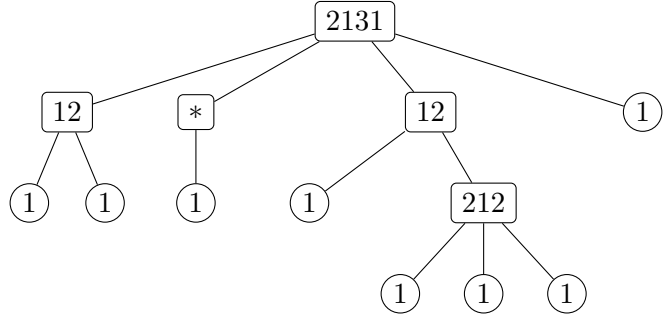

\begin{theorem}\label{finite-basis}
  A class $\C$ of Cayley permutations with only finitely many simple members has a
  finite basis.
\end{theorem}

\begin{proof}
  If $\C$ contains no nonempty word, its basis consists of the empty word
  or the word $1$. Otherwise, let $k$ be the greatest length of a simple
  member of $\C$. By Lemma~\ref{closure-simples} and Theorem~\ref{pwo},
  the substitution closure $\cl{\C}$ is partially well ordered.

  The basis elements of $\C$ that lie in $\cl{\C}$ form an antichain
  there, so there are finitely many. Any remaining basis element is also
  a basis element of $\cl{\C}$: it lies outside $\cl{\C}$,
  while all its proper patterns lie in $\C\subseteq\cl{\C}$.
  Corollary~\ref{closure-basis} bounds its length by $k+2$, leaving only
  finitely many possibilities.
\end{proof}

It remains to find the generating function. As in
Example~\ref{algebraic-class}, we must keep track of avoidance when
inflating a quotient. A finite basis makes this possible with
finitely many generating functions by recording which patterns below a
basis element occur in each component. This is analogous to the
query-complete property sets used by Brignall, Huczynska, and
Vatter~\cite{BHV08} for ordinary permutations: the properties satisfied
by an inflation are determined by its quotient and the properties
satisfied by its components. Before giving the construction, we
state the algebraic fact needed to pass from the equations to the result.

\begin{lemma}\label{algebraic-system}
  Let $y_j=P_j(x,y_1,\ldots,y_m)$, $1\leq j\leq m$, be a
  \emph{proper algebraic system} over $\QQ$; that is,
  $P_1,\ldots,P_m\in\QQ[x,y_1,\ldots,y_m]$ satisfy
  \begin{align*}
    P_j(0,\ldots,0)&=0,
      &&1\leq j\leq m,\\
    \frac{\partial P_j}{\partial y_\ell}(0,\ldots,0)&=0,
      &&1\leq j,\ell\leq m.
  \end{align*}
  Then the system
  \[
    y_j=P_j(x,y_1,\ldots,y_m),\qquad\qquad\quad 1\leq j\leq m,
  \]
  has a unique solution
  $(f_1,\ldots,f_m)$ with every $f_j\in x\mskip1mu\QQ[[x]]$, and each $f_j$ is
  algebraic over $\QQ(x)$. Moreover, given
  $H\in\QQ[x,y_1,\ldots,y_m]$, a nonzero polynomial in $\QQ[x,T]$
  annihilating $H(x,f_1,\ldots,f_m)$ can be computed from $H$ and the $P_j$.
\end{lemma}

\begin{proof}
  Existence and uniqueness of the solution with zero constant terms,
  and algebraicity of its components, follow from
  Stanley's~\cite[Proposition~6.6.3 and Theorem~6.6.10]{Sta99} results
  on proper algebraic systems.

  Panholzer~\cite{Pan05} gives an algorithm computing an annihilating
  polynomial for any component of a proper algebraic system. To apply it
  to $H$, write $y=(y_1,\ldots,y_m)$, put $c=H(0,\ldots,0)$, and append
  the equation
  \[
    z=H\bigl(x,P_1(x,y),\ldots,P_m(x,y)\bigr)-c.
  \]
  Its right-hand side vanishes at the origin, as do all its derivatives
  with respect to the unknowns, so the enlarged system is still proper.
  Its new component is $g=H(x,f_1,\ldots,f_m)-c$. Panholzer's algorithm
  therefore computes a nonzero $q(x,T)\in\QQ[x,T]$ with $q(x,g)=0$.
  Then $q(x,T-c)$ is the required polynomial annihilating
  $H(x,f_1,\ldots,f_m)$.
\end{proof}

\begin{theorem}\label{algebraicity}
  Let $\C$ be a class of Cayley permutations with only finitely many simple
  members. Its ordinary generating function
  $G(x)=\sum_{\omega\in\C}x^{|\omega|}$ is algebraic over $\QQ(x)$.
  Given the simple members and the basis of $\C$, one can compute a nonzero
  polynomial $q(x,T)\in\QQ[x,T]$ such that $q(x,G(x))=0$.
\end{theorem}

\begin{proof}
  If $\C$ contains no nonempty word, then $G(x)$ is $0$ or $1$,
  and the result is immediate. Otherwise, both
  the empty word and $1$ belong to $\C$.

  We construct a proper algebraic system that counts the nonempty members
  of $\cl{\C}$ according to the patterns they contain. We can then
  recover $G$ by selecting the members that avoid the basis of $\C$.
  Let $\mathcal B$ be the basis of $\C$, finite by Theorem~\ref{finite-basis},
  and let $\Gamma$ consist of all patterns contained in a member of
  $\mathcal B$. This is a finite set, closed downward. For a word $\omega$, let
  \[
    \mathcal P(\omega)=\{\gamma\in\Gamma:\gamma\leq \omega\}
  \]
  be the set of patterns in $\Gamma$ contained in $\omega$, called its
  \emph{profile}. This is analogous
  to the set of satisfied properties recorded by Brignall, Huczynska,
  and Vatter~\cite{BHV08}, with the properties here being containment
  of the individual patterns in $\Gamma$.
  Since $\mathcal B\subseteq\Gamma$, a word $\omega\in\cl{\C}$ belongs
  to $\C$ if and only if $\mathcal P(\omega)\cap\mathcal B=\emptyset$.

  The set $\mathcal P(\omega)$ is determined by the quotient of $\omega$ and the
  corresponding sets for its components. To see this, let
  $\omega=\sigma[\alpha_1,\ldots,\alpha_k]$. By Lemma~\ref{patterns-through},
  a nonempty $\beta\in\Gamma$ occurs in $\omega$ if and only if
  \[
    \beta=\st(\sigma|_T)[\gamma_i:i\in T]
  \]
  for nonempty $T\subseteq[k]$ and nonempty $\gamma_i\leq\alpha_i$,
  with $\gamma_i$ constant whenever $\sigma(i)$ is repeated. Each $\gamma_i$ is
  a pattern of $\beta$, so it lies in $\Gamma$. Thus the condition
  $\gamma_i\leq\alpha_i$ can be read from $\mathcal P(\alpha_i)$.
  Enumerating the finitely many choices of $T$ and the $\gamma_i\in\Gamma$
  determines whether each $\beta\in\Gamma$ occurs, and hence determines
  $\mathcal P(\omega)$. The empty
  pattern, if in $\Gamma$, belongs to every $\mathcal P(\omega)$.

  For each subset $Q\subseteq\Gamma$, define the ordinary generating function
  \[
    g_Q(x)=\sum_{\substack{\omega\in\cl{\C}\cap\Cay_+\\
                           \mathcal P(\omega)=Q}}x^{|\omega|}.
  \]
  Define $g_Q^+(x)$ and $g_Q^-(x)$ by restricting this sum to plus
  indecomposable and minus indecomposable words, respectively.
  Since $\Gamma$ is finite,
  this gives a finite list of generating functions.
  We also need the
  series $c_Q$ for the nonempty constant words in the closure with
  profile $Q$. These are rational: if $L$ is the greatest length in
  $\Gamma$, taking $L=1$ if $\Gamma$ is empty, then
  $\mathcal P(1^r)$ is constant for $r\geq L$. Hence each $c_Q$ is a
  finite sum of powers of $x$, possibly together with $x^L/(1-x)$.
  If $11$ is absent from the closure, only $1$ is available instead.

  We obtain equations for the $g_Q,g_Q^+,g_Q^-$ from the unique
  decomposition in Theorem~\ref{module-substitution-thm}.
  Constant words are counted by $c_Q$ and contribute to all three
  series, since they are both plus and minus indecomposable.
  By Lemma~\ref{closure-simples}, the available nonconstant simple
  quotients are exactly those in $\C$. For such a quotient $\sigma$
  of length $k\geq2$, choose component profiles $Q_1,\ldots,Q_k$.
  At a repeated quotient letter use the factor $c_{Q_i}$; at a singleton
  fibre use $g_{Q_i}$, with two exceptions: the first factor is
  $g_{Q_1}^+$ when $\sigma=12$ and $g_{Q_1}^-$ when $\sigma=21$.
  These are exactly the restrictions that make the
  decomposition unique.

  The product of these $k$ factors counts the inflations with the chosen
  property sets. Compute the resulting set $Q=\mathcal P(\omega)$
  as above, and include this product in the equation for $g_Q$.
  To obtain the equation for $g_Q^+$, omit the
  quotient $12$; to obtain that for $g_Q^-$, omit $21$. Indeed, these
  quotients characterize plus and minus decomposability, respectively.
  There are finitely many quotients and profiles, so this gives
  a finite system polynomial in the unknown series, with rational
  functions of $x$ as coefficients regular at the origin.
  Each right-hand side vanishes at the origin and has zero Jacobian
  there. Indeed, every inflation contribution has at least two nonempty parts,
  and all the factors counting such parts have zero constant term.

  Clear denominators, normalized to have constant term $1$, and move
  their nonconstant multiples of the unknowns to the right. This gives
  a proper polynomial system. Lemma~\ref{algebraic-system} therefore applies,
  and the counting series form its
  unique solution with zero constant terms. The ordinary generating
  function of $\C$ is
  \[
    G(x)
    =1+\sum_{\substack{Q\subseteq\Gamma\\Q\cap\mathcal B=\emptyset}}g_Q(x),
  \]
  so it is algebraic. Finally, $\Gamma$ and all the polynomial equations
  are computable from the basis and simple members of $\C$ by finite
  enumeration. Applying the last assertion of Lemma~\ref{algebraic-system}
  to $1+\sum_{Q\cap\mathcal B=\emptyset}g_Q(x)$ produces the required
  polynomial $q(x,T)$.
\end{proof}

Every simple word is semisimple, so Theorem~\ref{pwo},
Theorem~\ref{finite-basis} and Theorem~\ref{algebraicity} apply in particular
to classes with finitely many semisimple members. Their simple members
are obtained by retaining the primitive words and $11$. In this case, the
profile construction can also impose~\eqref{eq:repeat-restriction} and count
$\clH{\C}$. Use its finitely many semisimple quotients of length
at least two, and at every repeated quotient letter require the
component profile $\mathcal P(1)$ and use the factor $x$.
The remaining word $1$ contributes $x$ to each of the three series
with profile $\mathcal P(1)$; the quotient $11$, when present, is
counted with the other quotients. The profile transitions and the
normalizations at $12$ and $21$ are unchanged. This gives a proper
polynomial system directly.

\begin{example}\label{algebraic-system-class}
  Let $\C=\Av(112,212,231)$. It is easy to see that there are no simple members
  of lengths four or five, so Corollary~\ref{two-lengths} shows that its
  nonempty simple members are
  \begin{equation}\label{example-simple-members}
    1,\quad 11,\quad 12,\quad 21,\quad 121.
  \end{equation}
  By Lemma~\ref{closure-simples}, these are also the nonempty simple members
  of the closure $\cl{\C}$.
  However, $\C$ contains infinitely many nonconstant semisimple members as
  every word $121^b$, with $b\geq1$, is semisimple and lies in $\C$.
  The class is not substitution-closed, since $21,12,1\in\C$ but
  $21[12,1]=231\notin\C$. Nevertheless, every hereditary subclass of $\C$ has finitely
  many simple members and therefore has a finite basis and an algebraic
  generating function by Theorems~\ref{finite-basis} and~\ref{algebraicity}.

  Write $\mathcal B=\{112,212,231\}$ and let
  $\Gamma=\{1,11,12,21,112,212,231\}$
  be its downward closure, omitting the empty pattern.
  The simple basis element $212$ does not belong to
  $\cl{\C}$, so every member of $\cl{\C}$ avoids it. Thus
  $\omega\in\cl{\C}$ belongs to $\C$ if and only if
  $112,231\notin\mathcal P(\omega)$. We use the profile construction
  in the proof of Theorem~\ref{algebraicity}, summing
  over the sets $Q$ disjoint from $\mathcal B$ and distinguishing
  whether $11\in Q$ and whether $12\in Q$. Write
  \begin{align*}
    A&=\sum_{\substack{Q\subseteq\Gamma,\ Q\cap\mathcal B=\emptyset\\
                       12\notin Q}}g_Q,&
    B&=\sum_{\substack{Q\subseteq\Gamma,\ Q\cap\mathcal B=\emptyset\\
                       12\in Q}}g_Q,&
    I&=\sum_{\substack{Q\subseteq\Gamma,\ Q\cap\mathcal B=\emptyset\\
                       11\notin Q}}g_Q^+.
  \end{align*}
  Thus $A$ counts the nonempty members avoiding $12$, $B$ counts those
  containing $12$, and $I$ counts the plus indecomposable members avoiding
  $11$. Every word counted by $A$ is weakly decreasing and hence plus
  indecomposable. Its minus indecomposable members are the nonempty
  constant words, counted by $X=x/(1-x)$.
  The words counted by $I$ are precisely the plus indecomposable
  $231$-avoiding permutations, each of which consists of its maximum followed by
  an arbitrary $231$-avoiding permutation. Thus $I=x\,\mathrm{Cat}(x)$ and, using that $\mathrm{Cat}(x)=1+x\,\mathrm{Cat}(x)^2$,
  $$
  I=x\,\mathrm{Cat}(x)
  = x\bigl(1+x\,\mathrm{Cat}(x)^2\bigr)
  = x + x^2\,\mathrm{Cat}(x)^2
  = x + I^2.
  $$

  We now read the contributions quotient by quotient. The nonempty
  constant words contribute $X$ to $A$. For the quotient $12$, the
  first component is plus indecomposable and must avoid $11$, or else two equal
  letters there would
  form $112$ together with any letter of the second component. It is therefore counted by $I$.
  Since $231$ is plus indecomposable, an inflation of $12$ avoids $231$
  precisely when both components do. With the first component avoiding
  $11$, no occurrence of $112$ crosses the components either. Every such
  inflation contains $12$, giving the contribution $I(A+B)$ to $B$.

  For the quotient $21$, an occurrence of $231$ crosses the two components
  precisely when the first contains $12$. This component must also be
  minus indecomposable, so it is constant. No occurrence of $112$ crosses
  these components. The resulting inflation
  contains $12$ if and only if the second component does, giving $XA$ to
  $A$ and $XB$ to $B$.

  For the quotient $121$, the two positions in the repeated fibre are
  inflated by arbitrary nonempty constant words $1^a$ and $1^b$, whereas
  the middle component $\alpha$ is free. Avoidance of $112$ forces $a=1$.
  An occurrence of $231$ crosses
  the components of $121[1^a,\alpha,1^b]$
  precisely when $\alpha$ contains $12$. Thus $\alpha$ is counted by $A$,
  and these inflations contribute $xXA$ to $B$.

  Collecting the contributions, we obtain the system
  \begin{align*}
    A&=X+XA,\quad X=\frac{x}{1-x},\\
    B&=I(A+B)+XB+xXA,\\
    I&=x+I^2.
  \end{align*}
  Each right-hand side vanishes at the origin and has no term linear in
  an unknown alone. Clearing denominators gives a proper algebraic
  system, as in the proof of the theorem.

  To eliminate the auxiliary series, put $G=1+A+B$. The first equation
  gives $A=X/(1-X)$, and adding the first two gives
  $G-1=X+(I+X)(G-1)+xXA$. Hence
  \[
    G=1+\frac{X+xX^2/(1-X)}{1-X-I},
    \qquad I=\frac{1-\sqrt{1-4x}}{2}.
  \]
  This generating function is algebraic but not rational, since $I$ can
  be recovered rationally from $G$. Eliminating $I$ and clearing
  denominators produces the annihilating polynomial
  \begin{align*}
    q(x,T)={}&x^2(1-2x)^2(T-1)^2\\
      &-(1-x)^2(1-2x)(1-3x)(T-1)+x(1-x)^4.
  \end{align*}
  The solution with constant term $1$ begins
  \[
    G(x)=1+x+3x^2+10x^3+33x^4+109x^5+363x^6+1223x^7+\cdots
  \]

  For comparison, let $\mathcal D=\Av(111,112,212,231)$. A direct computation
  and Corollary~\ref{two-lengths} show that its nonempty semisimple members
  are exactly the five simple words listed in~\eqref{example-simple-members}. The same calculation now
  has $X=x+x^2$, since the nonempty constant words are $1$ and $11$,
  and the quotient $121$ contributes $x^2A$, since both outer components
  must be $1$. The series $I$ is unchanged. Thus the generating function
  of $\mathcal D$ is
  \[
    G_{\mathcal D}=1+\frac{(x+x^2)(1-x)}
      {(1-x-x^2)(1-x-x^2-I)},
  \]
  and begins
  \[
    G_{\mathcal D}(x)=1+x+3x^2+9x^3+28x^4+88x^5+284x^6+937x^7+\cdots
  \]
  Neither counting sequence appears in the OEIS~\cite{oeis}.
\end{example}

An exhaustive computation over the $2^{13}-1=8191$ nonempty bases
$\mathcal B\subseteq\Cay_3$ finds exactly $7380$ for which
$\Av(\mathcal B)$ has no simple member of length six or seven.
By Corollary~\ref{two-lengths}, all simple members of these classes
have length at most five. By Theorems~\ref{finite-basis}
and~\ref{algebraicity}, every hereditary subclass of any of these classes
has a finite basis and an algebraic generating function.

\section*{Acknowledgements}

We thank Christian Bean and Robert Brignall for valuable feedback on an
earlier version of this article.

The large language models GPT (OpenAI) and Claude (Anthropic)
were used during the preparation of this manuscript to assist with exploring and
checking conjectures through programming, drafting some propositions
and their proofs, light editing, and proofreading.

\end{document}